\documentclass[12pt,a4paper]{article}
\usepackage[latin1]{inputenc}
\usepackage{amsmath}
\usepackage{amsthm}
\usepackage{amssymb}
\usepackage{fullpage}
\theoremstyle{plain}

\newtheorem{theorem}{Theorem}[section]
\newtheorem{definition}{Definition}[section]
\newtheorem{lemma}{Lemma}[section]

\newtheorem{corollary}{Corollary}[section]
\newtheorem{example}{Example}[section]
\numberwithin{equation}{section}
\begin{document}
\begin{center} {\bf Legendre theorems for a class of partitions with initial repetitions}\end{center}
%\medskip
\begin{center}
 Darlison Nyirenda$^{1}$%\footnote{Corresponding Author: Darlison.Nyirenda@wits.ac.za} 
and
Beaullah Mugwangwavari$^{2}$ %\footnote{712040@students.wits.ac.za}
 \vspace{0.5cm} \\
$^{1}$The John Knopfmacher Centre for Applicable Analysis and Number Theory, University of the Witwatersrand, P.O. Wits 2050, Johannesburg, South Africa.\\
$^{2}$ School of Mathematics, University of the Witwatersrand, P. O. Wits 2050, Johannesburg, South Africa.\\
e-mails: darlison.nyirenda@wits.ac.za, 712040@students.wits.ac.za\\

\end{center}
\begin{abstract}
Partitions with initial repetitions were introduced by George Andrews. We consider a subclass of these partitions and find Legendre theorems associated with their respective partition functions. The results in turn provide partition theoretic interpretations of some Rogers-Ramanujan identities due to Lucy J. Slater.
\end{abstract}
\section{Introduction}
A partition of $n$ is a non-increasing sequence  of positive integers: $(\lambda_1,\lambda_2, \lambda_3, \cdots, \lambda_s)$ such that $\sum\limits_{i = 1}^{s}\lambda_i  = n$. The summands $\lambda_i$'s are called \textit{parts} and the \textit{length} of a partition is the total number of parts (counting multiplicity). Instead of the \lq vector notation\rq\,, we sometimes use the multiplicity notation $(\mu_1^{m_1},\mu_2^{m_2}, \mu_3^{m_3}, \cdots, \mu_\ell^{m_\ell})$ in which $m_i $ denotes the multiplicity of the part $\mu_i$ and $\mu_1 > \mu_2 > \cdots > \mu_{\ell}$. If $m_i = 1$ for all $i$, we have a partition into distinct parts. \\
The \textit{union} of two partitions $\lambda$ and $\beta$, denoted by $\lambda \cup \beta$, is simply the multiset union  where $\lambda$ and $\beta$ are treated as multisets. For instance, if $\lambda = (9^3, 7^2, 1^3)$ and $\beta = (7^4, 5^3, 4, 1^3)$, then $\lambda \cup \beta = (9^3, 7^6, 5^3, 4, 1^6)$. \\
\noindent For a partition $\lambda = (\mu_1^{m_1},\mu_2^{m_2}, \mu_3^{m_3}, \cdots, \mu_\ell^{m_\ell})$,  the conjugate of $\lambda$ is denoted by $\lambda^{\prime}$  and it is given by:
$$ \lambda^{\prime} = \left( \left( \sum\limits_{i = 1}^{r} m_i \right)^{\mu_{\ell}}, \left(\sum\limits_{i = 1}^{\ell-1} m_i \right)^{\mu_{\ell - 1} - \mu_{\ell}}, \cdots, m_{1}^{\mu_{1} - \mu_{2}}\right).$$
\noindent We shall use upper case letters for sets and lower case for counting functions. If $A(n)$ denotes the set of partitions of $n$ with  a certain property, then $a(n)$ denotes the cardinality of $A(n)$, i.e. $a(n) = |A(n)|$. An element of $A(n)$ shall be referred to as an $a(n)$-partition.\\
We shall use $D(n)$ to denote the set of partitions of $n$ into distinct parts and so by our notation above, $d(n) = |D(n)|$. For instance, $d(5) = 3$ and the $d(5)$-partitions are $(5), (4,1)$ and $(3,2)$.  Let $d_{e}(n)$ (resp. $d_{o}(n)$) be the number of $d(n)$-partitions with even (resp. odd) length. Legendre \cite{legendre} proved that  
\begin{equation}\label{legend}
d_{e}(n) - d_{o}(n)  =
\begin{cases}
(-1)^{j}, & \text{if $n = j(3j \pm 1)/2, j \geq 0$;}\\
0,& \text{otherwise}.
\end{cases}
\end{equation}
a result that later became known as Legendre's theorem.  Note that \eqref{legend} is also known as the pentagonal number theorem because the numbers $j(3j\pm 1)/2, j\in \mathbb{Z}$ are the generalized pentagonal numbers. In fact, the numbers appearing in all our theorems in this paper are generalized polygonal numbers (or multiples thereof). \\
An interesting bijective proof  of \eqref{legend} was given by J. Franklin (see \cite{andrews}). For related work in partition theory, see \cite{Nyirenda1, nyire1, nyire2, nyire3, banda, Nyirenda2, Nyirenda3}.\\
\noindent Fine \cite{fine} studied partitions without gaps. A partition without gaps whose parts are in the set $A$ is one in which every part is in $A$ and  every positive integer that is less than the largest part appears as a part. For instance, for $A = \{1 + 2j: j = 0,  1,2, 3, \ldots \}$, we can talk about partitions into odd parts without gaps. We shall also use the terminology \textit{gap-free partitions} to mean partitions without gaps. If the set $A$ is not explicitly stated, we assume that $A = \{1,2,3,4,5, \ldots \}$.\\
\noindent  Motivated by Fine's work on partitions into odd parts without gaps and an observation that partitions without gaps are in one-to-one correspondence with  partitions into distinct parts, George Andrews \cite{initial} introduced partitions with initial repetitions. His definition is given as follows:
\begin{definition}\label{defwani}
A partition of $n$ with initial $k$-repetitions  is one in which either
\begin{enumerate}
\item[a)] every part appears at most $k - 1$ times \\ or
\item[b)] there is some part $j$ which appears at least $k$ times and every positive integer less than $j$ appears at least $k$ times as a part, and all parts greater than $j$ appear at most $k - 1$ times.
\end{enumerate}
\end{definition}
\noindent For example, the partition $ (10^2, 7, 4^{3}, 3^{3},2^{5}, 1^{4})$ is a partition with initial 3-repetitions.\\\\
Among several results, in the same paper \cite{initial}, Andrews was able to show that if  $f_{e}(m,n)$ (rep. $f_{o}(m,n)$) denotes the number of partitions of $n$ with initial 2-repetitions  with $m$ different parts and an even (resp. odd) number of distinct parts, then
$$
f_{e}(m,n) - f_{o}(m,n) = 
\begin{cases}
(-1)^{j}, & \text{if $m = j, n = j(j + 1)/2, j \geq 0$;}\\
0,& \text{otherwise}.
\end{cases}
$$
\noindent Statements of the type \eqref{legend} are called Legendre theorems or identities of Euler pentagonal type.\\
\noindent In this paper, we find Legendre theorems associated with partitions with initial repetitions. Our first consideration is Andrews' partitions with initial 2-repetitions. We look at subsets of these partitions in Section 2 and derive an interesting identity. In Section 3, we find several Legendre theorems. Consequently, these theorems provide partition-theoretic interpretations of the following identities of Rogers-Ramanujan type due to Slater \cite{slater}:
\begin{equation}\label{eq7}
\prod\limits_{n = 1}^{\infty} (1-q^{n}) \sum\limits_{n = 0}^{\infty} \frac{q^{n(n+1)}}{(q^{2};q^{2})_{n}} = \prod\limits_{n = 1}^{\infty}( 1 - q^{4n})(1 - q^{4n-1})(1 - q^{4n - 3})
\end{equation}
\begin{equation}\label{eq9}
\prod\limits_{n = 1}^{\infty} (1-q^{2n}) \sum\limits_{n = 0}^{\infty} \frac{q^{n(2n+1)}}{(q;q)_{2n+1}} = \prod\limits_{n = 1}^{\infty}( 1 - q^{4n})(1 + q^{4n-1})(1 + q^{4n - 3})
\end{equation}
\begin{equation}\label{eq14}
\prod\limits_{n = 1}^{\infty} (1-q^{n}) \sum\limits_{n = 0}^{\infty} \frac{q^{n(n+1)}}{(q;q)_{n}} = \prod\limits_{n = 1}^{\infty}( 1 - q^{5n})(1 - q^{5n-1})(1 - q^{5n - 4})
\end{equation}
\begin{equation}\label{eq38}
\prod\limits_{n = 1}^{\infty} (1-q^{2n}) \sum\limits_{n = 0}^{\infty} \frac{q^{2n(n+1)}}{(q;q)_{2n+1}} = \prod\limits_{n = 1}^{\infty}( 1 - q^{8n})(1 + q^{8n-1})(1 + q^{8n - 7})
\end{equation}
\begin{equation}\label{eq53}
\prod\limits_{n = 1}^{\infty} (1-q^{4n}) \sum\limits_{n = 0}^{\infty} \frac{q^{4n^{2}}(q;q^{2})_{2n}}{(q^{4};q^{4})_{2n}} = \prod\limits_{n = 1}^{\infty}( 1 - q^{12n})(1 - q^{12n-5})(1 - q^{12n-7})
\end{equation}
\begin{equation}\label{eq55}
\prod\limits_{n = 1}^{\infty} (1-q^{4n}) \sum\limits_{n = 0}^{\infty} \frac{q^{4n(n+1)}(q;q^{2})_{2n+1}}{(q^{4};q^{4})_{2n+1}} = \prod\limits_{n = 1}^{\infty}( 1 - q^{12n})(1 - q^{12n-1})(1 - q^{12n-11})
\end{equation}
\begin{equation}\label{eq57}
\prod\limits_{n = 1}^{\infty} (1-q^{4n}) \sum\limits_{n = 0}^{\infty} \frac{q^{4n(n+1)}(-q;q^{2})_{2n+1}}{(q^{4};q^{4})_{2n+1}} = \prod\limits_{n = 1}^{\infty}( 1 - q^{12n})(1 + q^{12n-1})(1 + q^{12n-11})
\end{equation}
\begin{equation}\label{eq70}
\prod\limits_{n = 1}^{\infty} \frac{(1-q^{2n})}{(1+q^{2n-1})} \sum\limits_{n = 0}^{\infty} \frac{q^{n(n+2)}(-q;q^{2})_{n+1}(-q^{2};q^{4})_{n}}{(q^{2};q^{2})_{2n+1}} = \prod\limits_{n = 1}^{\infty}( 1 - q^{16n})(1 - q^{16n-4})(1 + q^{16n-12})
\end{equation}
Throughout  our discussion, we assume that $|q|<1$ and some of the tools that we use include the following identities:
 \begin{equation}\label{jacobi}
\sum\limits_{n = -\infty}^{\infty}z^{n}q^{n(n+1)/2} = \prod\limits_{n = 1}^{\infty}( 1 - q^{n})(1 + zq^{n})(1 + z^{-1}q^{n - 1}) 
\end{equation}
where  $z \neq 0$  and
\begin{equation}\label{gauss}
\sum_{n = -\infty}^{\infty}(-1)^{n}q^{n^2} = \prod_{n = 1}^{\infty}\frac{1 - q^{n}}{1 + q^{n}}. 
\end{equation}
For reference, see Theorem 11 of \cite{andrewserik} and Corollary 2.10 of \cite{andrews}.
\section{On Andrews' partitions with 2-initial repetitions}
We first record the following result.
\begin{lemma}
For $|q|< 1$, the following factorisations hold:
\begin{equation}\label{factor1}
\sum_{n = 1}^{\infty}\frac{q^{2n^{2} - n}}{(q;q)_{2n - 1}} = (-q;q)_{\infty} \sum_{n = 1}^{\infty}(-1)^{n+1}q^{n^{2}},
\end{equation}
\begin{equation}\label{factor2}
\sum_{n = 0}^{\infty}\frac{q^{2n^{2} + n}}{(q;q)_{2n}} = (-q;q)_{\infty} \sum_{n = 0}^{\infty}(-1)^{n}q^{n^{2}}.
\end{equation}
\end{lemma}
\begin{proof} 
Since $d(n) =  d_{e}(n) - d_{o}(n) + 2d_{o}(n)$, we have
$$\sum_{n = 0}^{\infty}d(n)q^{n} = \sum_{n = 0}^{\infty}(d_{e}(n) - d_{o}(n))q^{n} + 2\sum_{n = 0}^{\infty}d_{o}(n)q^{n}.$$ 

 \begin{align*}
 2\sum_{n = 0}^{\infty}d_{o}(n)q^{n}   & = (-q;q)_{\infty} - \sum_{n = 0}^{\infty}(d_{e}(n) - d_{o}(n))q^{n}\\
                                            & = (-q;q)_{\infty} -  (q; q)_{\infty}\\
                                            & = (-q;q)_{\infty}\left(1 - \frac{(q;q)_{\infty}}{(-q;q)_{\infty}}\right)\\
                                            & = (-q;q)_{\infty}\left(1 - \sum_{n = -\infty}^{\infty}(-1)^{n}q^{n^{2}}\right)\,\,(\text{by}\,\,\eqref{gauss})\\
                                            & = 2(-q;q)_{\infty}\sum_{n = 1}^{\infty}(-1)^{n+1}q^{n^{2}}.                                                                                                                 
                                            \end{align*}
Thus it is not difficult to see that 
$$\sum_{n = 0}^{\infty}d_{o}(n)q^{n} = \sum_{n = 1}^{\infty}\frac{q^{n(2n - 1)}}{ (q;q)_{2n - 1}}$$
from which \eqref{factor1} follows. For \eqref{factor2}, we have
\begin{align*}
\sum_{n = 0}^{\infty}d_{e}(n)q^{n}  & = \sum_{n = 0}^{\infty}d(n)q^{n} - \sum_{n = 0}^{\infty}d_{o}(n)q^{n}\\
                                       & = (-q;q)_{\infty} + (-q;q)_{\infty}\sum_{n = 1}^{\infty}(-1)^{n}q^{n^{2}}\,\,(\text{by}\,\,\eqref{factor1})\\
                                       & = (-q;q)_{\infty}\sum_{n = 0}^{\infty}(-1)^{n}q^{n^{2}}.
\end{align*}
It can easily be shown that
$$\sum_{n = 0}^{\infty}d_{e}(n)q^{n} = \sum_{n = 0}^{\infty}\frac{q^{n(2n + 1)}}{(q;q)_{2n}}$$
and so  \eqref{factor2} follows.
\end{proof}
We prove the following theorem.
   \begin{theorem}\label{main1}
Let $b^{e}(n)$ be the number of partitions of $n$ with initial 2-repetitions in which  either all parts are distinct or the largest repeated part is even. Similarly, let $b^{o}(n)$ denote the number of partitions of $n$ with initial 2-repetitions in which at least one part is repeated and the largest repeated part is odd. Then
$$b^{e}(n) - b^{o}(n)
= \begin{cases}
1, & \text{if}\,\, n=\frac{j(j+1)}{2}; \\
0\, & \text{otherwise}.
\end{cases} $$
\end{theorem}
\begin{proof}
Note that 
\begin{align*}
\sum_{n = 0}^{\infty}b^{e}(n)q^{n}  & = \prod_{j  = 1}^{\infty}(1 + q^{j}) + \sum_{m = 1}^{\infty}\frac{q^{2(1 + 2 + 3 + \ldots + 2m)}}{(q;q)_{2m}} \prod_{j = 2m + 1}^{\infty}(1 + q^{j}) \\
                                   & =  \sum_{m = 0}^{\infty}\frac{q^{2(1 + 2 + 3 + \ldots + 2m)}}{(q;q)_{2m}} \prod_{j = 2m + 1}^{\infty}(1 + q^{j}) \\
                                   & = \sum_{m = 0}^{\infty}\frac{q^{2m(2m + 1)}}{(q;q)_{2m}} \prod_{j = 2m + 1}^{\infty}\frac{(1 - q^{2j})}{(1 - q^{j})} \\
                                  & =  \sum_{m = 0}^{\infty}\frac{q^{2m(2m + 1)}}{(q;q)_{2m} (q^{2m + 1};q)_{\infty}} \prod_{j = 2m + 1}^{\infty}(1 - q^{2j}) \\
                                   & = \frac{1}{(q;q)_{\infty}}\sum_{m = 0}^{\infty}q^{2m(2m + 1)} \frac{\prod_{j = 1}^{\infty}(1 - q^{2j})}{\prod_{j = 1}^{2m} (1 - q^{2j})}\\
                                   & = \frac{(q^{2};q^{2})_{\infty}}{(q;q)_{\infty}}\sum_{m = 0}^{\infty}\frac{q^{2m(2m + 1)}}{(q^{2};q^{2})_{2m}}
                                   \end{align*}
and 
\begin{align*}
\sum\limits_{n = 0}^{\infty}b^{o}(n)q^{n} & = \sum\limits_{m = 1}^{\infty}\frac{q^{2(1 + 2 + 3 + \ldots + 2m - 1})}{(q;q)_{2m - 1}} \prod_{j = 2m}^{\infty}(1 + q^{j}) \\
                                   &  =  \sum_{m = 1}^{\infty}\frac{q^{2(1 + 2 + 3 + \ldots + 2m - 1)}}{(q;q)_{2m - 1}} \prod_{j = 2m}^{\infty}\frac{1 - q^{2j}}{1 - q^{j}} \\
                                   & =  \sum_{m = 1}^{\infty}\frac{q^{2m(2m -1)}}{(q;q)_{2m - 1} (q^{2m};q)_{\infty}} \prod_{j = 2m}^{\infty}(1 - q^{2j}) \\
                                   & = \frac{1}{{(q;q)_{\infty}}} \sum_{m = 1}^{\infty} q^{2m(2m -1)} \frac{ \prod_{j = 1}^{\infty}(1 - q^{2j}) }{ \prod_{j = 1}^{2m - 1}(1 - q^{2j})}\\
                                   & = \frac{(q^{2};q^{2})_{\infty}}{(q;q)_{\infty}}\sum_{m = 1}^{\infty}\frac{q^{2m(2m -1)}}{(q^{2};q^{2})_{2m - 1}}.
                                   \end{align*}
\noindent Thus
\begin{align*} 
\sum\limits_{n = 0}^{\infty}(b^{e}(n) - b^{o}(n))q^{n} &  = \frac{(q^{2};q^{2})_{\infty}}{(q;q)_{\infty}}\left( \sum_{n = 0}^{\infty}\frac{q^{2n(2n + 1)}}{(q^{2};q^{2})_{2n}} - \sum_{n = 1}^{\infty}\frac{q^{2n(2n - 1)}}{(q^{2};q^{2})_{2n - 1}}\right)\\
                                            & = \frac{(q^{4};q^{4})_{\infty}}{(q;q)_{\infty}}\left( \sum_{n = 0}^{\infty}(-1)^{n}q^{2n^{2}} - \sum_{n = 1}^{\infty}(-1)^{n + 1}q^{2n^{2}}  \right)\\
                                                & \hspace{5mm} \left(\text{by}\,\,\eqref{factor1}\,\,\text{and}\,\, \eqref{factor2}\,\,\text{with $q$ replaced by $q^2$}\right)\\
                                                & = \frac{(q^{4};q^{4})_{\infty}}{(q;q)_{\infty}}\left(\sum_{n = -\infty}^{\infty}(-1)^{n}q^{2n^{2}}  \right)\\
                                                & = \frac{(q^{4};q^{4})_{\infty}}{(q;q)_{\infty}} \frac{(q^{2};q^{2})_{\infty}}{(-q^{2};q^{2})_{\infty}}\,\,(\text{by}\,\, \eqref{gauss} \,\,\text{with $q$ replaced by $q^2$})\\
                                                & =\frac{(q^{4};q^{4})_{\infty} (-q;q)_{\infty} }{(-q^{2};q^{2})_{\infty}}\\
                                            & = (q^{4};q^{4})_{\infty}(-q;q^{4})_{\infty}(-q^{3};q^{4})_{\infty}\\
                                            & = \sum_{n = -\infty}^{\infty}q^{2n^2 + n} \,\,(\text{by}\,\, \eqref{jacobi}\,\,\text{with $q$ replaced by $q^4$ and $z$ replaced by $q^{-1}$})\\
                                            & = \sum_{n = 0}^{\infty}q^{n(n+1)/2}.
\end{align*}
For the last equality, see Equation (1.4.9) of \cite{hirs}.
\end{proof}
\section{Combinatorial proof of Theorem \ref{main1}}
Let $B(n)$ be the set of partitions of $n$ with initial $2$-repetitions. If $\lambda\in B(n)$, write $\lambda=\alpha\cup \beta$ with $\alpha$ a distinct partition and $\beta$ a partition whose parts have even multiplicity. Then $\beta$ is gap-free and $\beta'$ (the conjugate of $\beta$) is a distinct partition with even parts. The goal is to prove that \begin{align*}|\{\lambda=\alpha\cup \beta\in B(n): &\  \ell(\beta') \text{ even}\}|-  |\{\lambda=\alpha\cup \beta\in B(n): \ell(\beta') \text{ odd}\}|\\ & =\begin{cases}1, & \text{ if } n=\frac{j(j+1)}{2}\\ 0 & \text{ otherwise. }\end{cases}\end{align*} We further write $\alpha= (\alpha^o, \alpha^e)$, where $\alpha^o$ (respectively $\alpha^e$) consists of the odd (respectively even) parts of $\alpha$. 
\noindent Since $(\alpha^e, \beta)$ is a pair of distinct partitions with even parts, by doubling all parts in partitions in the proof of \cite[Proposition 2]{BW}, one proves that 
\begin{align*}|\{\lambda& =\alpha^o\cup \alpha^e\cup \beta\in B(n):  \ell(\beta') \text{ even}\}|-  |\{\lambda=\alpha^o\cup \alpha^e\cup \beta\in B(n): \ell(\beta') \text{ odd}\}|\\ & =|\{\lambda=\alpha^o\cup \gamma\in C(n):   \ell(\gamma) \text{ even}\}|-  |\{\lambda=\alpha^o\cup \gamma \in C(n): \ell(\gamma) \text{ odd}\}|, \end{align*} where $C(n)$ is the set of partitions $\mu=\alpha^o\cup \gamma$ of $n$ with $\alpha^o$ a partition into distinct odd pats and $\gamma$ a partition into parts divisible by $4$. 

This shows combinatorially that $$\sum_{n=0}^\infty(b^e(n)-b^o(n))q^n=(q^4;q^4)_\infty(-q;q^2)_\infty= (q^4;q^4)_\infty(-q;q^4)_\infty(-q^3;q^4)_\infty.$$
To finish the combinatorial proof, one uses a combinatorial proof of the Jacobi triple product. For example, one can use the proof in \cite{KK}. Note that \cite{KK} gives a combinatorial proof for  $$(q^4;q^4)_\infty(q;q^4)_\infty(q^3;q^4)_\infty=\sum_{n=0}^\infty (-1)^{T_n}q^{T_n},$$ where $T_n=n(n+1)/2$, but the parity of the number of odd parts in a partition is determined by the parity of the size. 

\section{Related partition functions}
In this section, we explore various partition functions for partitions with initial repetitions.  We give several Legendre theorems  and as a result, partition-theoretic interpretation of equations \eqref{eq7}, \eqref{eq9}, \eqref{eq14}, \eqref{eq38}, \eqref{eq53}, \eqref{eq55}, \eqref{eq57} and \eqref{eq70} are established.\\\\
\noindent Let $c_{1}(n)$ denote the number of partitions of $n$ in which
 either
\begin{enumerate}
\item[(a)] all parts are distinct, the only odd part that may appear is 1 and even parts are at least 8 and divisible by 4  \\
or 
\item[(b)] the largest repeated even part $2j$ appears exactly  4 times, all positive even integers  $< 2j$ appear exactly 4 times, even parts $ > 2j$ are at least $8j + 8$, distinct and divisible by 4, odd parts are distinct and at most $4j + 1$.
\end{enumerate}
\noindent Let $c_{1,e}(n)$ (resp. $c_{1,o}(n)$) denote the number of $c_{1}(n)$-partitions in which the number of distinct even parts is even ( resp. odd). Then we have
\begin{theorem}
For all $n\geq 0$,
$$
c_{1,e}(n) - c_{1,o}(n) =
\begin{cases}
1, & \text{if $n = j(6j + 5), j\in \mathbb{Z}$};\\
0, & \text{otherwise}.
\end{cases}
$$
\end{theorem}
\begin{proof}
Since 
$$ \sum\limits_{n = 0}^{\infty}c_{1}(n)q^{n} = \sum_{n = 0}^{\infty}q^{4(2 + 4 + 6 + \ldots + 2n)}(-q;q^{2})_{2n + 1}(-q^{8n+8};q^{4})_{\infty},$$ we must have 
\begin{align*}
\sum\limits_{n = 0}^{\infty}(c_{1,e}(n) - c_{1,o}(n)) q^{n} & = \sum_{n = 0}^{\infty}q^{4(2 + 4 + 6 + \ldots + 2n)}(-q;q^{2})_{2n + 1}(q^{8n+8};q^{4})_{\infty} \\
                                                                                    & =  \sum_{n = 0}^{\infty}q^{4n(n + 1)}(-q;q^{2})_{2n + 1}(q^{4(2n+2)};q^{4})_{\infty} \\
                                                                                   & = \sum_{n = 0}^{\infty}q^{4n(n + 1)}(-q;q^{2})_{2n + 1}\frac{(q^{4};q^{4})_{\infty}}{ (q^{4};q^{4})_{2n + 1}}\\
 & = (q^{4};q^{4})_{\infty}\sum_{n = 0}^{\infty}(-q;q^{2})_{2n + 1} \frac{q^{4n(n + 1)}}{(q^{4};q^{4})_{2n + 1}} \\
& = \prod_{n = 1}^{\infty}(1 + q^{12n - 11})(1 + q^{12n - 1})(1 - q^{12n})\,\,\,\,\,(\text{by}\,\,\,\, \eqref{eq57}) \\
& = \sum_{n = -\infty}^{\infty}q^{n(6n + 5)} \,\,\,\,\,\,(\text{by} \,\,\,\,\eqref{jacobi}) .\\
\end{align*}
\end{proof}
\noindent Let $c_{2}(n)$ be the number of partitions of $n$ in which either
\begin{enumerate}
\item[(a)]  even parts are distinct and 1 is the only odd part that may appear\\
or
\item[(b)] there exists $j \geq 1$ such that an even part $2j$ appears twice, all positive integers $<2j$ appear twice, any even part $> 2j$ is distinct and the largest odd part is at most $2j+1$.\\ 
\end{enumerate}
Furthermore, let  $c_{2,e}(n)$ (resp. $c_{2,o}(n)$) be the number of $c_{2}(n)$-partitions with an even (resp. odd) number of distinct even parts. Then we have the following.
\begin{theorem}
For all $n\geq 0$,
$$c_{2,e}(n) - c_{2,o}(n)=
\begin{cases}
1, & n = 4j^2 + 3j, j \in \mathbb{Z}; \\
0, &\text{otherwise}.
\end{cases}
$$
\end{theorem}
Since  
\begin{align*}
\sum_{n = 0}^{\infty}c_{2}(n)q^{n}  & = \frac{ (-q^{2};q^{2})_{\infty}}{1 - q} + \sum_{n = 1}^{\infty}\frac{q^{2 + 2 + 4 + 4 + 6 + 6 + \cdots + 2n + 2n} }{(1 - q)(1 - q^{3})\ldots (1 - q^{2n+1})}(-q^{2n + 2};q^{2})_{\infty}\\
                                                         & = \sum_{n = 0}^{\infty}\frac{q^{2 + 2 + 4 + 4 + 6 + 6 + \cdots + 2n + 2n} }{(1 - q)(1 - q^{3})\ldots (1 - q^{2n+1})}(-q^{2n + 2};q^{2})_{\infty},
\end{align*} 
we have
\begin{align*}
\sum_{n = 0}^{\infty}(c_{2,e}(n) - c_{2,o}(n))q^{n}  & =  \sum_{n = 0}^{\infty} \dfrac{q^{2n(n+1)}(q^{2n+2};q^{2})_{\infty}}{(q;q^{2})_{n+1}} \\
& = \sum_{n = 0}^{\infty} \dfrac{q^{2n(n+1)}(q^{2};q^{2})_{\infty}}{(q;q^{2})_{n+1}(q^{2};q^{2})_{n}}  \\
& = (q^{2};q^{2})_{\infty} \sum_{n = 0}^{\infty} \dfrac{q^{2n(n+1)}}{(q;q^{2})_{n+1}(q^{2};q^{2})_{n}} \\
& = (q^{2};q^{2})_{\infty} \sum_{n = 0}^{\infty} \dfrac{q^{2n(n+1)}}{(q;q)_{2n+1}}  \,\,\,(\text{by}\,\,\,\, \eqref{eq38})\\
& = \prod_{n=1}^{\infty} \left(1 + q^{8n-1}\right)\left(1 + q^{8n-7}\right)\left(1-q^{8n}\right) \\
& = \sum_{n = -\infty}^{\infty} q^{4n^{2}+3n}.
\end{align*}
\begin{example}
Consider $n=10$. 
\end{example} 
\noindent The $c_{2}(10)$-partitions are: $$(10), (8,2), (8,1^{2}), (6,4), (6,2^{2}), (6,2,1^{2}), (6,1^{4}), (4,2^{2},1^{2}), (4,2,1^{4}), (4,1^{6}), (3^{2},2^{2}), (3,2^{2},1^{3}),$$ $$ (2^{2},1^{6}), (2,1^{8}), (1^{10})$$ 
From the above, note that $c_{2,e}(10)$-partitions are: $$(8,2), (6,4), (6,2,1^{2}), (4,2,1^4), (3^{2},2^{2}), (3,2^{2},1^{3}),(2^{2},1^{6}), (1^{10})$$ and $c_{2,o}(10)$-partitions are: $(10), (8,1^{2}), (6,2^2), (6,1^{4}), (4,2^2,1^2), (4,1^{6}), (2,1^{8})$. Thus, $$ c_{2,e}(10) - c_{2,o}(10) = 1.$$
This agrees with the theorem because the only integer solution to $4j^2 + 3j = 10$ is $j = -2$\\\\
\noindent Let $c_{3}(n)$ denote the number of partitions of $n$ in which the largest  odd part $2j + 1 (j\geq 0)$ occurs at least $j$ times, even parts are distinct and greater than $2j + 1$. \\
Note that $c_{3}(n)$-partitions are a subset of the set of partitions of $n$ with odd parts below even parts. Partitions with parts separated by parity have received quite a bit of attention lately, see \cite{andrewsparity, darlisonparity}.\\
\noindent Let  $c_{3,e}(n)$ (resp. $c_{3,o}(n)$) be the number of $c_{3}(n)$-partitions with an even (resp. odd) number of even parts. Then:
\begin{theorem}
For all $n\geq 0$,
$$c_{3,e}(n) - c_{3,o}(n) =
\begin{cases}
1, & n = 2j^2 + j, j \in \mathbb{Z}; \\\\
0, &\text{otherwise}.
\end{cases}
$$
\end{theorem}
\begin{proof}
We have
\begin{align*}
\sum_{n = 0}^{\infty}c_{3}(n)q^{n}  & = \sum_{n = 0}^{\infty}\frac{q^{\overbrace{(2n + 1) + (2n + 1) + (2n+ 1) + \ldots + (2n + 1)}^{n\rm \, times}}}{(1 - q)(1 - q^{3})\ldots (1 - q^{2n+1})}(-q^{2n + 2};q^{2})_{\infty}  
\end{align*} 
and thus,
\begin{align*}
\sum_{n = 0}^{\infty}(c_{3,e}(n) - c_{3,o}(n))q^{n}  & = \sum_{n = 0}^{\infty}\frac{q^{\overbrace{(2n + 1) + (2n + 1) + (2n+ 1) + \ldots + (2n + 1)}^{n\rm \, times}}}{(1 - q)(1 - q^{3})\ldots (1 - q^{2n+1})}(q^{2n + 2};q^{2})_{\infty}  \\
                                                                               & = \sum_{n = 0}^{\infty} \dfrac{q^{n(2n+1)}(q^{2n+2};q^{2})_{\infty}}{(q;q^{2})_{n+1}} \\
& = \sum_{n = 0}^{\infty} \dfrac{q^{n(2n+1)}(q^{2};q^{2})_{\infty}}{(q;q^{2})_{n+1}(q^{2};q^{2})_{n}} \\
& = (q^{2};q^{2})_{\infty} \sum_{n = 0}^{\infty} \dfrac{q^{n(2n+1)}}{(q;q)_{2n+1}} \\
 & = \prod_{n = 1}^{\infty}(1 + q^{4n - 1})(1 + q^{4n - 3})(1 - q^{4n}) \,\,\,(\text{by}\,\,\,\, \eqref{eq9}) \\
& = \sum_{n = -\infty}^{\infty} q^{2n^{2} + n}.
\end{align*} 
\end{proof}
\noindent Let $c_{4}(n)$ be the number of partitions of $n$ in which either
\begin{enumerate} 
\item[(a)] all parts are distinct \\
or 
\item[(b)] the largest repeated part $j$ appears twice, all positive integers less than  $j$ appear twice. Note that parts greater than $j$ are distinct.
\end{enumerate}
Similar to the previous formulations, let $c_{4,e}(n)$ (resp.$c_{4,o}(n)$) be the number of $c_{4}(n)$-partitions with an even (resp. odd) number of distinct parts.
\begin{theorem}
For all $n\geq 0$,
$$c_{4,e}(n) - c_{4,o}(n)=
\begin{cases}
(-1)^{j}, & n = (5j^2 + 3j)/2, j \in \mathbb{Z}; \\\\\
0, &\text{otherwise}.
\end{cases}
$$
\end{theorem}
\begin{proof}
Clearly, 
$$\sum_{n = 0}^{\infty} c_{4}(n)q^{n}  = \sum_{n = 0}^{\infty} q^{1 + 1 + 2 + 2 + 3 + 3 + \cdots + n + n}(-q^{n+1};q)_{\infty}  = \sum_{n = 0}^{\infty} q^{n(n + 1)}(-q^{n+1};q)_{\infty} $$
so that
\begin{align*}
\sum_{n = 0}^{\infty}(c_{4,e}(n) - c_{4,o}(n))q^{n} & = \sum_{n = 0}^{\infty} q^{n(n+1)}(q^{n+1};q)_{\infty} \\
& = \sum_{n = 0}^{\infty} \dfrac{q^{n(n+1)}(q;q)_{\infty}}{(q;q)_{n}} \\
& = \prod_{n=1}^{\infty} \left(1-q^{5n-1}\right)\left(1-q^{5n-4}\right)\left(1-q^{5n}\right)  \,\,\,(\text{by}\,\,\,\, \eqref{eq14}) \\
& = \sum_{n = -\infty}^{\infty} (-1)^{n} q^{\frac{5n^{2}+3n}{2}}.
\end{align*}
\end{proof}

\begin{example}
Consider $n=7$. 
\end{example}
\noindent The $c_{4}(7)$-partitions are:
$$ (7), (6,1), (5,2), (5,1^{2}), (4,3), (4,2,1), (3,2,1^{2}).$$
$c_{4,e}(7)$-partitions are: $$ (6,1), (5,2), (4,3), (3,2,1^{2})$$
and the $c_{4,o}(7)$-partitions are: $$(7),  (5,1^{2}), (4,2,1).$$
Thus, $c_{4,e}(7) - c_{4,o}(7) = 1$. Indeed, this verifies the theorem as $7 = [5(-2)^2 + 3(-2)]/2$. Note that $j = -2$ is the only integer solution to the equation $7 = (5j^2 + 3j)/2$. \\

\noindent Let $c_{5}(n)$ be the number of partitions of $n$ in which either
\begin{enumerate}
\item[(a)] all parts are distinct  \\
or
\item[(b)]  there exists $j \geq 1$ such that all odd positive integers $\leq j$ appear twice or thrice and other odd parts are distinct, all even positive integers $\leq j$ appear twice, even parts $>2j$ are distinct and no even integer in the interval $(j, 2j]$ appears.
\end{enumerate}
\noindent Let $c_{5,e}(n)$ (resp. $c_{5,o}(n)$)  denote the number of $c_{5}(n)$-partitions with an even (resp. odd) number of distinct even parts.  Then

\begin{align*} \sum_{n = 0}^{\infty} c_{5}(n)q^{n}  & = \sum_{n = 0}^{\infty} q^{1 + 1 + 2 +  2 + \cdots  n + n}(-q^{2n+2};q^{2})_{\infty}(-q;q^{2})_{\infty} \\
                                                                              & = \sum_{n = 0}^{\infty} q^{n(n+1)}(-q^{2n+2};q^{2})_{\infty}(-q;q^{2})_{\infty}
\end{align*} 
and 
$$ \sum_{n = 0}^{\infty}(c_{5,e}(n) - c_{5,o}(n)) q^{n} = \sum_{n = 0}^{\infty} q^{n(n+1)}(q^{2n+2};q^{2})_{\infty}(-q;q^{2})_{\infty}.$$
Observe that
\begin{align*}
\sum_{n = 0}^{\infty}(c_{5,e}(n) - c_{5,o}(n)) (-q)^{n} & = \sum_{n = 0}^{\infty} (-q)^{n(n+1)}((-q)^{2n+2};q^{2})_{\infty}(-(-q);(-q)^{2})_{\infty}\\
& = (q;q^{2})_{\infty} \sum_{n = 0}^{\infty} q^{n(n+1)}(q^{2n+2};q^{2})_{\infty} \\
& = (q;q^{2})_{\infty} \sum_{n = 0}^{\infty} \dfrac{q^{n(n+1)}(q^{2};q^{2})_{\infty}}{(q^{2};q^{2})_{n}} \\
& = (q;q)_{\infty} \sum_{n = 0}^{\infty} \dfrac{q^{n(n+1)}}{(q^{2};q^{2})_{n}} \\
& = \prod_{n=1}^{\infty} \left(1-q^{4n-1}\right)\left(1-q^{4n-3}\right)\left(1-q^{4n}\right)  \,\,\,(\text{by}\,\,\,\, \eqref{eq7}) \\
& = \sum_{n = -\infty}^{\infty}(-1)^{n} q^{2n^{2}+n}.
\end{align*}
We have the following result:
\begin{theorem}\label{partt7}
For all $n\geq 0$,
$$ c_{5,e}(n) - c_{5,o}(n)
= \begin{cases}
1, & n = 2j^{2}+ j, j \in \mathbb{Z};\\\\
0, & \text{otherwise}.
\end{cases}
$$
\end{theorem}

\noindent Let $c_{6}(n)$ be the number of partitions of $n$ in which either
\begin{enumerate}
\item[(a)] all parts are even, distinct and divisible by 4 \\
or
\item[(b)]  the largest repeated part is  $2j  - 1$ (for some $j \geq 1$) and appears exactly 4 times or 5 times, all positive odd integers $< 2j - 1$  appear 4  times or 5 times, any other odd part is distinct and is at most $4j - 1$ in  part size, even parts are $\geq 8j + 4$, distinct and divisible by 4.
\end{enumerate}
\noindent Let $c_{6,e}(n)$ (resp. $c_{6,o}(n)$)  denote the number of $c_{6}(n)$-partitions with an even (resp. odd) number of distinct even parts.  Then
$$ \sum_{n = 0}^{\infty} c_{6}(n)q^{n}  = \sum_{n = 0}^{\infty} q^{4(1 + 3 + 5 + \cdots + 2n - 1)}(-q^{8n+4};q^{4})_{\infty}(-q;q^{2})_{2n}$$ so that
$$ \sum_{n = 0}^{\infty} (c_{6,e}(n) -  c_{6,o}(n))q^{n}  = \sum_{n = 0}^{\infty} q^{4n^{2}}(q^{8n+4};q^{4})_{\infty}(-q;q^{2})_{2n}$$ which implies
\begin{align*}
\sum_{n = 0}^{\infty}(-1)^{n} (c_{6,e}(n) -  c_{6,o}(n))q^{n} & = \sum_{n = 0}^{\infty} q^{4n^{2}}(q^{8n+4};q^{4})_{\infty}(q;q^{2})_{2n} \\
& = \sum_{n = 0}^{\infty} \dfrac{q^{4n^{2}}(q^{4};q^{4})_{\infty}(q;q^{2})_{2n}}{(q^{4};q^{4})_{2n}} \\
& = (q^{4};q^{4})_{\infty} \sum_{n = 0}^{\infty} \dfrac{q^{4n^{2}}(q;q^{2})_{2n}}{(q^{4};q^{4})_{2n}} \\
& = \prod_{n=1}^{\infty} \left(1-q^{12n-5}\right)\left(1-q^{12n-7}\right)\left(1-q^{12n}\right) \,\,\,(\text{by}\,\,\,\, \eqref{eq53}) \\
& = \sum_{n = -\infty}^{\infty} (-1)^{n}q^{6n^{2}+n}.
\end{align*}
Hence, we have:

\begin{theorem}\label{partt7}
For all $n\geq 0$,
$$ c_{6,e}(n) - c_{6,o}(n)
= \begin{cases}
1, & n = 6j^{2}+ j, j \in \mathbb{Z};\\\\
0, & \text{otherwise}. 
\end{cases}
$$
\end{theorem}

\noindent Let $c_{7}(n)$ denote the number of partitions of $n$ in which either
\begin{enumerate} 
\item[(a)] the smallest odd part, if it appears, is at least $3$ and even parts are distinct and at least $4$ \\
or 
\item[(b)] there exists $j \geq 1$ such that $j$ appears three or four times if $j \equiv 2 \pmod{4}$ and  appears exactly three times if $j \not \equiv 2 \pmod{4}$, all positive integers $i < j$   appear exactly twice or thrice if $i \equiv 2 \pmod{4}$ and appear exactly twice if $i \not \equiv 2 \pmod{4}$, an even part greater than $j$ but less than $4j$ is distinct and congruent to $ 2 \pmod{4}$, any other even part is $\geq 4j + 4$ and distinct, odd parts $> j$ are at least $2j + 3$.  
\end{enumerate}
Let $c_{7,e}(n)$ (resp. $c_{7,o}(n)$) be the number of $c_{7}(n)$-partitions in which the number of distinct even parts if $(a)$ holds is even (resp. odd) or the number of distinct even parts that are $\ge 4j+4$  (where $j$ is the largest repeated part with the property that every integer less than $j$ is repeated) is even (resp. odd) if $(b)$ holds. Clearly,
\begin{align*}
\sum_{n = 0}^{\infty} c_{7}(n)q^{n}  & = \frac{(-q^{4};q^{2})_{\infty}}{(q^{3};q^{2})_\infty} + \sum_{n = 1}^{\infty} \dfrac{q^{1 + 1 + 2 + 2 + \cdots + n - 1 + n - 1 + n + n + n}(-q^{2};q^{4})_{n}(-q^{4n+4};q^{2})_{\infty}}{(q^{2n+3};q^{2})_{\infty}} \\
                                                            & = \sum_{n = 0}^{\infty} \dfrac{q^{n(n+2)}(-q^{2};q^{4})_{n}(-q^{4n+4};q^{2})_{\infty}}{(q^{2n+3};q^{2})_{\infty}}
\end{align*}
and thus
\begin{align*}
\sum_{n = 0}^{\infty}( c_{7,e}(n) - c_{7,o}(n)  )q^{n} & = \sum_{n = 0}^{\infty} \dfrac{q^{n(n+2)}(-q^{2};q^{4})_{n}(q^{4n+4};q^{2})_{\infty}}{(q^{2n+3};q^{2})_{\infty}}\\
 & = \sum_{n = 0}^{\infty} \dfrac{q^{n(n+2)}(-q^{2};q^{4})_{n}(q;q^{2})_{n+1}(q^{2};q^{2})_{\infty}}{(q;q^{2})_{\infty}(q^{2};q^{2})_{2n+1}} \\
& = \dfrac{(q^{2};q^{2})_{\infty}}{(q;q^{2})_{\infty}} \sum_{n = 0}^{\infty} \dfrac{q^{n(n+2)}(-q^{2};q^{4})_{n}(q;q^{2})_{n+1}}{(q^{2};q^{2})_{2n+1}} \\
& = \prod_{n=1}^{\infty} \left(1-q^{16n-4}\right)\left(1-q^{16n-12}\right)\left(1-q^{16n}\right)  \,\,\,(\text{by}\,\,\,\, \eqref{eq70}) \\
& = \sum_{n = -\infty}^{\infty} (-1)^{n}q^{8n^{2}+4n}.
\end{align*}

\noindent This leads to the theorem below.

\begin{theorem}
For all $n\geq 0$,
$$c_{7,e}(n) - c_{7,o}(n) =
\begin{cases}
(-1)^{j}, & n = 8j^2 + 4j, j \in \mathbb{Z}; \\\\\
0, &\text{otherwise}.
\end{cases}
$$
\end{theorem}
\noindent From the theorem, it can be observed that, if $n\neq 8j^2 + 4j$ for all $j$, then $c_{7,e}(n) - c_{7,o}(n)  = 0$ so that 
$$c_{7}(n) \equiv c_{7,e}(n) - c_{7,o}(n)  \equiv 0 \pmod{2}.$$
\noindent We record this result below.
\begin{corollary}
 If $n$ is not four times a triangular number, then $c_{7}(n) \equiv 0 \pmod{2}$.
\end{corollary}
\noindent Let $c_{8}(n)$ be the number of partitions of $n$ in which either
\begin{enumerate}
\item[(a)] even parts are distinct, greater than 6 and divisible by 4 and the only odd part that may appear is 1 and is distinct. \\
or
\item[(b)]  the largest repeated part is  $2j$ (for some $j \geq 1$) and appears exactly 4 times, all positive even integers $< 2j$  appear four times, any even part $> 2j$ is at least $8j + 8$, distinct and divisible by 4, odd parts are distinct and are at most $4j +1$ in  part size.
\end{enumerate}
Let $c_{8,e}(n)$ (resp. $c_{8,o}(n)$)  denote the number of $c_{8}(n)$-partitions with an even (resp. odd) number of distinct even parts.  Then
$$ \sum_{n = 0}^{\infty} c_{8}(n)q^{n}  = \sum_{n = 0}^{\infty} q^{4(2 + 4 + 6 + \cdots + 2n)}(-q^{8n+8};q^{4})_{\infty}(-q;q^{2})_{2n + 1}$$ so that
$$ \sum_{n = 0}^{\infty} (c_{8,e}(n) -  c_{8,o}(n))q^{n}  = \sum_{n = 0}^{\infty} q^{4n(n + 1)}(q^{8n+8};q^{4})_{\infty}(-q;q^{2})_{2n + 1}$$ which implies
\begin{align*}
\sum_{n = 0}^{\infty} (-1)^{n}(c_{8,e}(n) -  c_{8,o}(n))q^{n} & = \sum_{n = 0}^{\infty} q^{4n(n+1)}(q^{8n+8};q^{4})_{\infty}(q;q^{2})_{2n+1}\\
& = \sum_{n = 0}^{\infty} \dfrac{q^{4n(n+1)}(q^{4};q^{4})_{\infty}(q;q^{2})_{2n+1}}{(q^{4};q^{4})_{2n+1}} \\
& = (q^{4};q^{4})_{\infty} \sum_{n = 0}^{\infty} \dfrac{q^{4n(n+1)}(q;q^{2})_{2n+1}}{(q^{4};q^{4})_{2n+1}}  \\
& = \prod_{n=1}^{\infty} \left(1-q^{12n-1}\right)\left(1-q^{12n-11}\right)\left(1-q^{12n}\right) \,\,\,(\text{by}\,\,\,\, \eqref{eq55}) \\
& = \sum_{n = -\infty}^{\infty} (-1)^{n}q^{6n^{2}+5n}.
\end{align*}
Hence, we have:

\begin{theorem}\label{partt7}
For all $n\geq 0$,
$$ c_{8,e}(n) - c_{8,o}(n)
= \begin{cases}
1, & n = 6j^{2}+ 5j, j \in \mathbb{Z};\\\\
0, & \text{otherwise}.
\end{cases}
$$
\end{theorem}
\begin{example}
Consider $n=39$. 
\end{example}
\noindent The $c_{8}(39)$-partitions are:
$$ (28,3,2^{4}), (9,5,4^{4},2^{4},1), (7,5,4^{4},3,2^{4}).$$
$c_{8,e}(39)$-partitions are: $$ (9,5,4^{4},2^{4},1), (7,5,4^{4},3,2^{4}) $$
and the $c_{8,o}(39)$-partitions are: $$(28,3,2^{4}).$$
Thus, $c_{8,e}(39) - c_{8,o}(39) = 1$. Indeed, this verifies the theorem as $39 = 6(-3)^2 + 5(-3)$ and $39 -3 = 36$. Note that $j = -3$ is the only integer solution to the equation $6j^2 + 5j = 39$.
%\section*{Declarations}
%\section{Declarations}
%\subsection{Competing Interests and/or Funding}
%On behalf of all authors, the corresponding author states that there is no conflict of interest. No funding was received for conducting this study. The authors have no financial or proprietary interests in any material discussed in this article. 
%\begin{itemize}
%\item Funding: No funding was received for conducting this study. The author has no financial interests in any material discussed in this article.
%\item Conflict of interest: The authors state that there is no conflict of interest.
%\item The authors confirm that this manuscript has no associated data.
%\item Authors' contributions: The authors contributed equally to this work.
%\end{itemize}

\end{document}